\documentclass{amsart}

\usepackage{packages}
\usepackage{formatting}
\usepackage{notation}

\title{\texorpdfstring{L-equivalences via symplectic and \(F_4\) Grassmannians}{L-equivalences via symplectic and F4 Grassmannians}}
\author{Ivan Noden}

\begin{document}

\begin{abstract}
  Using a construction of Kanemitsu from \cite{kanemitsu:2018} and 
  observations by Rampazzo in \cite{rampazzo:2021}, 
  we find examples of zero divisors in the Grothendieck ring of 
  varieties by taking the zero loci of sections of vector bundles over symplectic 
  and \( F_4 \) Grassmannians. These zero divisors yield instances of
  non-trivially L-equivalent Calabi-Yau varieties. This methodology is inspired
  by a similar process performed by Ito \textit{et al}. on \( G_2 \)
  Grassmannians in \cite{ito:2018}.
\end{abstract}

\maketitle

\setcounter{tocdepth}{1}
\tableofcontents

\section{Introduction}

The \emph{Grothendieck ring of varieties} over a field \( k \), 
denoted \( \GRVar{k} \) is, as an abelian group, generated by symbols of the 
form \( [X] \), where \( X \) is a variety over \( k \), subject
to the following relations:
\begin{enumerate}
  \item \( [X] = [Y] \) whenever \( X \cong Y \);
  \item \([X \setminus Y] = [X] - [Y]\) whenever \( Y \) is a closed
    subvariety of \(X\);
  \item \([X] \cdot [Y] = [X \times Y] \).
\end{enumerate}
This ring is algebraically quite interesting and has been useful
in various parts of algebraic geometry, for instance in Kontsevich's 
theory of motivic integration. Poonen showed in \cite{poonen:2002} that
\( \GRVar{k} \) is not a domain when \(k\) is characteristic zero and,
in \cite{borisov:2017}, Borisov proved that \(\L = [\A^1]\) is a zero
divisor in \(\GRVar{\C}\). Motivated by this, Ito \textit{et al}. constructed
a pair of Calabi-Yau varieties, \(X\) and \(Y\), cut out from 
\(G_2\) Grassmannians, such that
\begin{displaymath}
  \L ([X] - [Y]) = 0
\end{displaymath}
and \([X] \neq [Y] \) \cite{ito:2018}. In his PhD thesis \cite{rampazzo:2021},
Rampazzo observed that this construction is a particular
example of a family of identities in \(\GRVar{\C} \) obtained
by cutting out pairs of Calabi-Yau varieties from Grassmannians.
These constructions require a homogeneous roof, as
defined and studied by Kanemitsu in \cite{kanemitsu:2018}.
Kanemitsu identifies seven families of this construction:
\(A_r \times A_r, \, A_r^M, \, A_{2r}^G, \, C_{3r-1}, \, D_r, \, F_4\) and
\( G_2\). Rampazzo shows the latter five of these
have the potential to produce zero divisors (as already demonstrated
in the \(G_2\) case) and, in \cite{rampazzo:2022}, they are able to produce zero divisors
in the \(A_{2r}^G\) case. Rampazzo and Xie also provide the argument
for the \(D_5\) case in \cite{rampazzo_xie:2025}: it is a consequence
of \cite[Theorem 3.10]{rampazzo_xie:2025} and \Cref{cor:equals_iso}.
It is likely that this argument can be adapted to the general \(D_r\)
case for \(r \geq 5\).

The zero divisors constructed from these homogeneous roofs
yield examples of \emph{L-equivalent} varieties.

\begin{defn}[{\cite[\S 1.2]{kuznetsov_shinder:2018}}]
  Varieties \(X\) and \(Y\) are said to be \emph{L-equivalent}
  if there exists some \(d \in \Z_{\geq 0}\) such that
  \begin{displaymath}
    \L^d([X]-[Y])=0.
  \end{displaymath}
\end{defn}

There has been much evidence relating derived and L-equivalence.
In particular, the Calabi-Yau pairs from the roofs of type \(G_2\),
\(A_{4}^G\) and \(D_5\) are known to be both derived and L-equivalent 
\cite{ito:2018, kuznetsov:2018, rampazzo_kapustka:2019, rampazzo:2022, rampazzo_xie:2025}.
It was conjectured by Kuznetsov and Shinder
in \cite[Conjecture 1.6]{kuznetsov_shinder:2018} that any pair of projective, simply connected, smooth,
derived equivalent varieties are L-equivalent. However,
in \cite{meinsma:2025}, Meinsma found counter-examples to this conjecture
and instead conjectured the partial converse: that any pair of projective, 
hyperk\"{a}hler, L-equivalent manifolds are derived equivalent.
We hope that providing further examples of L-equivalent varieties
will assist in understanding the connection between derived and L-equivalence.

In this paper, we show the homogeneous roofs of type \(C_{3r-1}\) 
and \(F_4\) yield examples of L-equivalent varieties.
More precisely, for symplectic Grassmannians, we prove:
\begin{theorem}\label[theorem]{th:type_c_zero}
  Let \( r\geq 2\) and \((Z_1, Z_2)\) a Calabi-Yau pair associated
  to the homogeneous roof of type \(C_{3r-1}\). Then, \([Z_1] \neq [Z_2]\)
  and yet
  \begin{displaymath}
    \L^{2r-1}([Z_1]-[Z_2]) = 0.
  \end{displaymath}
\end{theorem}

For \(F_4\) Grassmannians we prove:
\begin{theorem}\label[theorem]{th:type_f4_zero}
  Let \((Z_1, Z_2)\) be a Calabi-Yau pair associated to the
  homogeneous roof of type \(F_4\). Then \([Z_1]\neq [Z_2]\)
  and yet
  \begin{displaymath}
    \L^{2}([Z_1]-[Z_2]) = 0.
  \end{displaymath}
\end{theorem}

The techniques of this paper are not enough to find
a zero divisor in the \(C_2\) case as, if \((Z_1, Z_2)\) is a Calabi-Yau
pair associated to a homogeneous roof of type \(C_2\), then the 
argument we use to show \([Z_1] \neq [Z_2]\)
does not follow due to dimensional constraints. More precisely,
we cannot use \Cref{lem:pic_restrict}.
Similarly, since \Cref{lem:pic_restrict} and \Cref{cor:equals_iso}
both fail in the \(D_4\) case, finding a zero divisor in this
instance will likely need different methods even if the results of
\cite{rampazzo_xie:2025} are generalised.

\section*{Acknowledgements}

This work was based on the author's Master's thesis,
supervised by Dario Beraldo, who I would like to thank
for all his valuable guidance on the project. I would also
like to thank Travis Schedler and Ed Segal for providing
corrections and comments on this paper.

This work was supported by the Engineering and Physical Sciences
Research Council [EP/Z534882/1].

\section{Preliminaries}

Let us briefly recap the representation theory of reductive algebraic 
groups in order to fix some notation. The reader should feel free
to skip this section and refer back to it when necessary. 
For a more in-depth and comprehensive discussion, we refer
the reader to the standard texts of Borel \cite{borel:1991},
Humphreys \cite{humphreys:1975}, Milne \cite{milne:2017} and Springer \cite{springer:1998}.
We work solely over \(\C\).

\subsection{Weights, Roots and the Weyl Group}

For an algebraic group \(G\), let \(\chars(G)\) denote the abelian
group of morphisms \(G \to \G_m\) under pointwise multiplication.
We write this group additively.

We fix a reductive group \(G\) with a maximal torus \(T\).
A representation \(V\) of \(G\) decomposes as a direct 
sum \cite[Theorem 3.2.3]{springer:1998}:
\begin{displaymath}
  V = \bigoplus_{\chi \in \chars(T)} V_\chi
\end{displaymath}
where 
\begin{displaymath}
  V_\chi = \set{v\in V}{tv = \chi(t)v \text{ for all } t \in T},
\end{displaymath}
is the weight space for the weight \(\chi\).
The \emph{roots} of \(G\) are the weights
of the adjoint representation and are denoted by \(\Phi\).

The \emph{Weyl group}
\begin{displaymath}
  W(G) = N_G(T)/T
\end{displaymath}
acts on the weights of a representation \(V\) of \(G\).
We write \(W\) for \(W(G)\) when there is
no ambiguity.

The \emph{cocharacters} of \(G\), denoted \(\cochars(G)\), is the 
group of morphisms \(\G_m \to G\) under pointwise
multiplication, again, written additively. 

Recall that \(\cochars(T)\) is a free abelian group of
finite rank and that we get a perfect pairing \(\chars(T)\times \cochars(T) \to \Z\)
by letting \(\innerprod{\chi}{\lambda} = r\), where \((\chi\circ\lambda)(x)=x^r\)
\cite[Lemma 3.2.11]{springer:1998}.

Each root \(\chi\) of \(G\) is assigned a \emph{coroot} \(\dual{\chi}\) 
as described in \cite[\S 7]{springer:1998}.
For \(\chi \in \Phi\), write
\begin{displaymath}
  s_\chi(\lambda) = \lambda - \innerprod{\lambda}{\dual{\chi}}\chi \in W.
\end{displaymath}

Recall, a \emph{system of positive roots} \(\Phi^+\) for \(\Phi\)
contains a set of \emph{simple roots} \(\mathcal{S}\)
such that \(W\) is generated by \(s_\chi\) 
for \(\chi \in \mathcal{S}\). These generators
are called \emph{simple reflections}.
A system of positive roots can be determined by
a choice of Borel subgroup of \(G\).

\subsection{Bruhat Decomposition}

Recall, for a Borel subgroup \(B\) of \(G\), \(B \times B\) acts on \(G\) by \((b_1, b_2)g = b_1gb_2^{-1}\).
Denote the orbit of \(g \in G\) under this action by \(BgB\).
This is a subvariety of \(G\). Note, for \(w \in W\),
the subvariety \(BwB\) is well-defined. 
The Bruhat decomposition \cite[Theorem 8.3.8]{springer:1998}
is the disjoint union
\begin{displaymath}
  G = \bigcup_{w\in W} BwB.
\end{displaymath}
Similarly, for \(G/B\), we have the
cellular decomposition
\begin{displaymath}
  G/B = \bigcup_{w\in W} X(w),
\end{displaymath}
where \(X(w)\) is some affine space and the union is disjoint. 
To determine its dimension,
we need to introduce some more notation. The \emph{length} of
\(w \in W\), denoted \(\ell(w)\), is the number of roots
in \(\Phi^+\) sent to \(-\Phi^+\) by \(w\). Equivalently,
\(w\) can be written as a product of \(\ell(w)\) simple
reflections and no fewer. There is a unique element of \(W\)
of length \(\#\Phi^+\) called the \emph{longest element}.
It is possible to show that
\begin{displaymath}
  X(w) \cong \A^{\ell(w)}.
\end{displaymath}
We therefore remark that \(W\) determines the cohomology of \(G/B\).
The proof of these claims and further details can be found
in \cite[Proposition 8.5.1]{springer:1998}.

Recall that parabolic subgroups of \(G\) can be constructed by taking
\(I \subseteq \mathcal{S}\), letting \(W_I \) be the subgroup
generated by \(s_\chi\) for \(\chi \in I\) and putting
\begin{displaymath}
  P_I = \bigcup_{w \in W_I} BwB.
\end{displaymath}
When we take a parabolic subgroup \(P\) of \(G\), we assume it is
of this form and call \(W_I\) the Weyl group of \(P\).

For each \(w\in W\), there is
a unique element of \(wW_I\) with the shortest possible length \cite[Proposition 1.1]{kock:1991}
so we put
\begin{displaymath}
  W^I = \set{w \in W}{\ell(w) \leq \ell(w^\prime) \text{ for all } w^\prime \in wW_I}.
\end{displaymath}
Then,
\begin{displaymath}
  G/P_I = \bigcup_{w\in W^I} Y(w)
\end{displaymath}
where the union is disjoint and \(Y(w) \cong \A^{\ell(w)}\) \cite[Proposition 1.3]{kock:1991}.

\subsection{Representations}

If \(V\) is a finite-dimensional, irreducible representation of
\(G\) then, up to isomorphism, \(V\) is uniquely determined
by its unique, highest weight. If \(\chi \in \chars(T)\)
is dominant then we let \(V_G(\chi)\) denote the
finite-dimensional, irreducible representation with highest
weight \(\chi\). For more detail on this, see \cite[\S 22]{milne:2017}.

We recall that, if \(\chi\) is the
highest weight of \(V\), then \(-\chi \) is the lowest weight of
\(\dual{V}\). Moreover, \(\dual{V} = V_G(-w_0(\chi)) \), where \(w_0\)
is the longest element of \(W\).

Also, recall that any finite-dimensional representation of \(G\)
is a direct sum of irreducible representations \cite[Theorem 22.42]{milne:2017}.

\subsection{The Semisimple Case}

Suppose \(G\) is semisimple, then the rank of \(\chars(T)\) is 
\(\#\mathcal{S}\) \cite[Proposition 21.48]{milne:2017}. Write \(\mathcal{S} = \{\alpha_1, \ldots, \alpha_r\} \).
Recall the \emph{fundamental weights} of \(G\),
\(\omega_1, \ldots, \omega_r\), are uniquely determined by the fact
\begin{displaymath}
  \innerprod{\omega_i}{\dual{\alpha}_j} = \begin{cases}
    1, & \text{if } i=j, \\
    0, & \text{otherwise.}
  \end{cases}
\end{displaymath}
Then the set of dominant weights of \(G\) is
\begin{displaymath}
  \set{\sum_{i=1}^r n_i \omega_i}{n_i \in \Z_{\geq 0}}.
\end{displaymath}

If \(G\) is semisimple, we can describe its parabolic subgroups
graphically using its Dynkin diagram. For \(I \subseteq \mathcal{S}\),
the parabolic \(P_I\) is represented by crossing out any vertex
of the Dynkin diagram corresponding to a simple root not in \(I\).

\subsection{Representations of Parabolic Subgroups}

Suppose \(G\) is now only assumed to be reductive and \(P\)
is a parabolic subgroup of \(G\). Recall, \(P\) contains
a Levi subgroup \(L\). Let \(V\) be a non-trivial, finite-dimensional,
irreducible representation of \(P\). Put \(U \) as the unipotent radical
of \(P\) so
\(P = U \rtimes L\). Then \(V\) is a representation of
\(U\) which is unipotent, thus the subspace of \(V\)
fixed by \(U\), denoted \(V^U\), is non-zero \cite[Proposition 14.3]{milne:2017}.
As \(U\) is normal in \(P\), \(V^U\) is a
\(P\)-subrepresentation of \(V\). Thus \(V = V^U\)
as \(V\) is irreducible. Hence, the action of \(P\)
on \(V\) is determined by the action of \(L\) on \(V\).
That is, finite-dimensional, irreducible representations
of \(P\) are in bijective correspondence with finite-dimensional,
irreducible representations of the reductive group \(L\) and,
hence, weights dominant for \(L\).

Note that, while a representation \(V\) of \(P\) cannot generally
be decomposed into a direct sum of irreducible representations,
it is possible if \(U\) acts trivially on \(V\). In particular,
it is possible if \(V\) is a subrepresentation of a tensor power of an 
irreducible representation.

Suppose \(P = P_I\) is a parabolic subgroup constructed
from a set of simple roots \(I\).
Then, a character \(\chi \in \chars(T)\) is dominant for \(L\) 
if and only if \(\innerprod{\chi}{\dual{\lambda}} \geq 0\) for
all \(\lambda \in I\). We say such a weight is \emph{dominant for} \(P\)
and let \(V_P(\chi)\) denote the finite-dimensional, irreducible
representation of \(P\) with highest weight \(\chi\).

Note if \(G\) is semisimple with simple roots \(\{\alpha_1, \ldots, \alpha_r\}\)
and corresponding fundamental weights \(\{\omega_1,\ldots, \omega_r\}\),
then the weights dominant for \(P\) are
\begin{displaymath}
  \set{\sum_{i=1}^r n_i \omega_i}{n_i \in \Z \text{ and } n_i \geq 0 \text{ whenever } \alpha_i \in I}.
\end{displaymath}
When studying representations of parabolic subgroups,
we will often refer to the Weyl group of \(P\).
By this, we mean the Weyl group of \(L\), seen as a subgroup
of \(W(G)\). This is the subgroup generated by the simple
reflections corresponding to the roots in \(I\).

\subsection{Homogeneous Vector Bundles}

If \(P\) is a parabolic subgroup of \(G\) then,
a \emph{homogeneous vector bundle} on \(G/P\)
is of the form
\begin{displaymath}
  E_P(V) = \frac{G \times V}{(g,v) \sim (gp, p^{-1}v) \text{ for all } p \in P},
\end{displaymath}
for \(V\) a representation of \(P\) \cite[Lemma 5.5.8, Theorem 15.1.3]{springer:1998}.
If \(\chi \in \chars(T)\) is dominant for
\(P\), we let \(E_P(\chi)\) denote \(E_P(\dual{V_P(\chi)})\).

Assuming \(G\) is semisimple and simply connected, every
line bundle on \(G/P\) is homogeneous \cite[Theorem 6.4]{snow} and hence line bundles
on \(G/P\) are in bijective correspondence with \(1\)-dimensional
representations of \(P\). Let 
\begin{displaymath}
  \Lambda_P = \set{\sum_{i=1}^r n_i \omega_i}{n_i \in \Z \text{ and } n_i = 0 \text{ whenever } \alpha_i \in I}
\end{displaymath}
and
\begin{displaymath}
  \Gamma_P = \set{\chi \in \chars(T)}{\chi \text{ is dominant for } P \text{ and } \dim V_P(\chi)=1}.
\end{displaymath}
We claim \(\Lambda_P = \Gamma_P\).

First, suppose \(\chi \in \Lambda_P\). Then \(\chi\) is an extremal
weight of \(V_P(\chi)\) so has multiplicity \(1\).
Consider the simple reflection \(s_{\alpha_i}\) for some simple
root \(\alpha_i\). Note that 
\begin{displaymath}
  s_{\alpha_i}(\omega_j) = \omega_j - \innerprod{\omega_j}{\dual{\alpha}_i}\alpha_i
  =
  \begin{cases}
    \omega_j, & \text{if } i \neq j, \\
    \omega_j - \alpha_i, & \text{if } i=j.
   \end{cases}
\end{displaymath}
Thus, every element of the Weyl group of \(P\) fixes \(\omega_i\)
if \(\alpha_i \in \mathcal{S} \setminus I\), hence every element
fixes \(\chi\) so it is the only extremal weight of \(V_P(\chi)\).
Thus, \(\dim V_P(\chi) = 1\) and \(\chi \in \Gamma_P\).

For the converse, suppose \(\chi \in \Gamma_P\). Then
we see \(\chi\) must be invariant under the Weyl group of \(P\).
This tells us that the highest weight of \(\dual{V_P(\chi)}\)
is \(-\chi\) and, hence, \(-\chi\) must be dominant for \(P\).
See that \(\chi\) and \(-\chi\) can only both be dominant
if \(\chi \in \Lambda_P\) so we get \(\Gamma_P \subseteq \Lambda_P\).

We therefore have a bijection between \(1\)-dimensional
representations of \(P\) (and hence line bundles on \(G/P\))
and \(\Lambda_P\).

Note that \(\Lambda_P\) is a group under addition
and is isomorphic to \(\Z^d\), where \(d = r - \# I\). It is
then an easy check to see \(\Pic(G/P) \cong \Lambda_P\).
Thus, the Picard number of \(G/P\) equals the number
of crossed out vertices in the Dynkin diagram of \(P\).
When we are in this case, let us write \(\OO_P(n_1, \ldots, n_r)\)
for the line bundle \(E_P(n_1\omega_{i_1}+ \cdots + n_r\omega_{i_r})\),
where \(\{\alpha_{i_1}, \ldots, \alpha_{i_r}\} = \mathcal{S} \setminus I\).
Due to our odd choice of convention for defining \(E_P(\chi)\),
this coincides with standard notation for line bundles on projective
space.

We finally remark that, by \cite[Corollary 8.5]{snow}, if \(G\) is semisimple and simply connected
then the homogeneous vector bundle \(E_P(\chi)\)
over \(G/P\) is ample if
\begin{displaymath}
  \chi \in \set{\sum_{i=1}^r n_i \omega_i}{n_i \in \Z_{\geq 0} \text{ and } n_i > 0 \text{ whenever } \alpha_i \not\in I}.
\end{displaymath}

\subsection{Borel-Weil-Bott}

Suppose now \(G\) is semisimple and simply connected. Let \(\rho\)
be the sum of the fundamental weights of \(G\). For \(w \in W\)
and \(\chi \in \chars(T)\) we define
\begin{displaymath}
  w \bullet \chi = w(\chi + \rho) - \rho.
\end{displaymath}

Let us recall the Borel-Weil-Bott theorem:
\begin{theorem}[Borel-Weil-Bott]
  Let \(P\) be a parabolic subgroup of \(G\) and \(\chi\in\chars(T)\)
dominant for \(P\). Either:
\begin{enumerate}
  \item there exists a unique \(w \in W\) such that \(w \bullet \chi\)
    is dominant for \(G\) and then
    \begin{displaymath}
      H^i(G/P, E_P(\chi)) \cong \begin{cases}
        0, & \text{if } i \neq \ell(w), \\
        \dual{V_G(w \bullet \chi)}, & \text{if } i = \ell(w);
      \end{cases}
    \end{displaymath}
  \item or there is no \(w \in W \) such that \(w \bullet \chi\) is dominant
    for \(G\) and then \(H^i(G/P, E_P(\chi)) = 0\) for all \(i\).
\end{enumerate}
\end{theorem}

\begin{proof}
  See, for instance, \cite[\S 5.1]{baston_eastwood:1989}.
\end{proof}

\begin{remark}\label[remark]{rmk:bwb}
  This result generalises to the right derived pushforward of other 
  quotient maps. 
  More precisely, suppose \(P\) and \(Q\) are parabolics of \(G\) with
  \(Q \subseteq P\). Put \(\map{\pi}{G/Q}{G/P}\) as the natural map.
  Then, for any \(\chi \in \chars(T)\) dominant for both \(P\) and \(Q\),
  we have \(\pi_\ast E_Q(\chi) \cong E_P(\chi) \), while \(R^i\pi_\ast E_Q(\chi) = 0\)
  for \(i > 0\). For more on this, see \cite[\S 5.3]{baston_eastwood:1989}.
\end{remark}

\subsection{Example: Symplectic Group}

We flesh out the case where \(G\) is a symplectic group
in order to fix notation used in \Cref{sec:symp}.

Let \(G = \Sp{2n}\) and \(T\) be the subgroup of \(G\) consisting
of diagonal matrices.
For \(i \in \{1, \ldots, n\}\), let \(L_i \in \chars(T)\) be 
defined such that \(L_i(X)\)
is the \(ii\)-entry of \(X\). The roots of \(G\) are
of the form \(L_i - L_j\), \(L_i + L_j\), \(-L_i - L_j\),
\(2L_i\) and \(-2L_i\) for \(i \neq j\). Choosing the Borel
subgroup \(B\) consisting of upper triangular matrices
yields a system of positive roots with simple roots
\(\alpha_i = L_i - L_{i+1}\) for \(i \in \{1, \ldots, n-1\}\)
and \(\alpha_n = 2L_n\). 
The fundamental weights
are \(\omega_i = L_1 +\cdots + L_i\).

The Dynkin diagram for \(G\) is \dynkin C{}, 
labelled \(\alpha_1, \ldots, \alpha_n\) from left to right.
Let \(s_i = s_{\alpha_i}\). Then the Weyl group \(W\) is
generated by \(s_1, \ldots, s_n\) and the action
on an element of \(\chars(T)\) is described as follows:
for \(i \in \{1, \ldots, n-1\}\), \(s_i\) swaps any instances of \(L_i\)
and \(L_{i-1}\), while \(s_n\) changes the sign of any instance of \(L_n\).

We denote by \(\IGr{r}{2n}\) the Grassmannian of \(r\)-dimensional
isotropic subspaces. Similarly, let us write
\(\Iflag{r_1, \ldots, r_t}{2n}\) for the space of isotropic
flags where the subspaces have dimensions \(r_1 < \cdots < r_t\).
Note that this flag variety is the quotient \(G/P\) where \(P\)
is the parabolic given by crossing out the \(r_1, \ldots, r_t\)
vertices of the Dynkin diagram of \(G\).

\section{Homogeneous Roofs}\label{section:homog}

Our L-equivalent varieties come from a construction of Kanemitsu in \cite{kanemitsu:2018}
known as a homogeneous roof.

\begin{defn}
  A \emph{roof} of rank \(r\) is a Fano variety \(M\) of Picard number
  \(2\) and index \(r\) equipped with two different \(\P^{r-1}\)-bundle
  structures. A roof is \emph{homogeneous} if \(M\) is
  isomorphic to some \(G/Q\) with \(G\) a semisimple
  algebraic group and \(Q\) a parabolic subgroup.
\end{defn}

In \cite[\S 5.2]{kanemitsu:2018}, Kanemitsu provides a
complete list of homogeneous roofs which we
replicate here:
\begin{center}
    \begin{tabular}{| c | c | c | c |}
        Type & Simply Connected Group & Dynkin Diagram for \( Q \) & Rank \\
        \hline
             &&& \\
        \( A_r \times A_r \) & \( \SL{r+1} \times \SL{r+1} \) &
        \dynkin[edge length = 2em, labels={1,,,r}]A{x*.**} \dynkin[edge length
        = 2em, labels={1,,,r}]A{x*.**} & \( r+1 \) \\[2em]
        \( A_r^M \) & \( \SL{r+1} \) & \dynkin[edge length = 2em, labels={1,,,r}]A{x*.*x} & \( r
        \) \\[2em]
        \( A_{2r}^G  (r \geq 2)\) & \( \SL{2r+1} \) &
        \dynkin[edge length = 2em, labels={1,,r,r+1,,2r}] A{*.*xx*.*} & \( r+1 \) \\[2em]
        \( C_{3r-1} \) & \( \Sp{6r-2} \) &
        \dynkin[edge length = 2em, labels={1,,2r-1,2r,,,3r-1}] C{*.*xx*.**} & \( 2r \) \\
        \( D_r (r\geq 4) \) & \( \Spin{2r} \) & \dynkin[edge length = 2em,
        labels={1,,,r-1,r}]D{*.**xx} & \( r \) \\
        \( F_4 \) & \( F_4 \) & \dynkin[edge length = 2em] F{*xx*} & \( 3 \) \\[2em]
        \( G_2 \) & \( G_2 \) & \dynkin[edge length = 2em] G{xx} & \( 2 \) \\
    \end{tabular}
\end{center}

The classification comes from choosing a semisimple,
simply connected algebraic group \(G\) and a parabolic
subgroup \(Q\) obtained by crossing out two vertices
of the Dynkin diagram, ensuring that \(G/Q\) has Picard number \(2\).
If we then let \(P_1\) and \(P_2\) be the parabolics 
obtained by crossing out just one of these vertices,
we require that the contractions \(\map{q_i}{G/Q}{G/P_i}\)
have fibres isomorphic to projective spaces of the same dimension,
which forces \(G\) and \(Q\) to be one of the above.
We call the \(G/P_i\) the \emph{bases} of the roof.

Put \(M = G/Q\) and \(F_i = G/P_i\). Note \(\Pic(F_i) \cong \Z\).
Write \(\OO_i\) for \(\OO_{P_i}\) and take \(\mathcal{E}\)
to be the line bundle \(q_1^\ast\OO_1(1) \otimes q_2^\ast\OO_2(1)\).
Write \(\mathcal{E}_i = {q_i}_\ast\mathcal{E}\) and observe it is
a vector bundle by \Cref{rmk:bwb}.

\begin{exmp}[Roof of type \(C_{3r-1}\)]
  For the roof of type \(C_{3r-1}\), the group is \(G = \Sp{6r-2}\),
  the homogeneous roof is \(M = \Iflag{2r-1,2r}{6r-2}\) and the bases
  are \(F_1 = \IGr{2r-1}{6r-2}\) and \(F_2 = \IGr{2r}{6r-2}\). 

  For instance, when \(r=1\), we have \(M = \Iflag{1,2}{4}\),
  \(F_1 = \P^3\) and \(F_2 = \IGr{2}{4}\). This
  roof is closely related to the Abuaf flop which Segal
  showed leads to a derived equivalence in \cite{segal:2016}. 
  The construction of this roof can be seen in 
  \cite[Section 2.1]{segal:2016}.
\end{exmp}

\begin{prop}[{\cite[Lemma 4.4.3]{rampazzo:2021}}] \label[proposition]{prop:proj_nice}
  Let \(M, F_i, \mathcal{E}\) and \(\mathcal{E}_i\) be as above.
  Then \(\P(\mathcal{E}_i)\cong M\).
\end{prop}

\begin{remark}
  As a consequence of the above, for a roof of rank \(r\), the bundles \(\mathcal{E}_i\) have rank \(r\).
\end{remark}

Let \(\sigma \in H^0(M, \mathcal{E})\) be a general section and take
\(T = Z(\sigma)\). By \({q_i}_\ast\sigma\) we denote the image
of \(\sigma\) under the isomorphism \(H^0(M,\mathcal{E}) \to H^0(F_i, \mathcal{E}_i)\).
Put \(Z_i = Z({q_i}_\ast \sigma)\).

\begin{lemma}[{\cite[Lemma 4.1.6]{rampazzo:2021}}]\label[lemma]{lem:Z_calabi}
  If \(M\) is a roof of rank \(r\), not of type \(A_{r-1} \times A_{r-1}\),
  then \(Z_i\) is a smooth, connected, Calabi-Yau variety of codimension
  \(r\).
\end{lemma}

\begin{remark}
  If \(M\) is of type \(A_r \times A_r\), then the \(Z_i\) are empty.
\end{remark}

In \cite[Definition 4.1.8]{rampazzo:2021}, Rampazzo calls
\((Z_1, Z_2)\) a \emph{Calabi-Yau pair} associated to \(M\).
It is these Calabi-Yau pairs that we will show to be
non-trivally L-equivalent in certain cases.

\begin{prop}[{\cite[\S 5.2]{rampazzo:2021}}] \label[proposition]{prop:the_relation}
  Let \(M\) be a homogeneous roof of rank \(r\) with bases \(F_1\) and \(F_2\).
  If \((Z_1, Z_2)\) is a Calabi-Yau pair associated to \(M\),
  then 
  \begin{displaymath}
    [\P^{r-2}]([F_2]-[F_1]) = \L^{r-1}([Z_1]-[Z_2]).
  \end{displaymath}
\end{prop}
\begin{proof}
  As observed by Rampazzo in \cite[\S 5.2]{rampazzo:2021}, if we let
  \(\overline{q}_i\) be the restriction of \(q_i\) to \(T\) then
  \begin{displaymath}
    \overline{q}_i^{-1}(x) = 
    \begin{cases}
      \P^{r-1}, & \text{ if } x \in Z_i, \\
      \P^{r-2}, & \text{ if } x \in F_i \setminus Z_i.
    \end{cases}
  \end{displaymath}
  Moreover, the \(\overline{q}_i\) are piecewise trivial fibrations
  \cite[p. 84]{rampazzo:2021} and thus 
  \begin{displaymath}
    [T] = [\P^{r-1}][Z_i] + [\P^{r-2}][F_i\setminus Z_i].
  \end{displaymath}
  As \([F_i\setminus Z_i] = [F_i]-[Z_i]\), we get
  \begin{align*}
    [\P^{r-2}]([F_2]-[F_1]) &= [\P^{r-2}]([Z_2]-[Z_1]) + [\P^{r-1}]([Z_1]-[Z_2]) \\
                            &= ([\P^{r-1}] - [\P^{r-2}])([Z_1]-[Z_2]).
  \end{align*}
  However, since
  \begin{displaymath}
    [\P^r] = 1 + \L + \cdots +\L^r,
  \end{displaymath}
  we have \([\P^{r-1}] - [\P^{r-2}] = \L^{r-1}\). Hence,
  \begin{displaymath}
    [\P^{r-2}]([F_2]-[F_1]) = \L^{r-1}([Z_1]-[Z_2]).
  \end{displaymath}
\end{proof}

From this, if we show \([F_1] = [F_2]\) and \([Z_1] \neq [Z_2]\),
we'll have
\begin{displaymath}
  \L^{r-1}([Z_1]-[Z_2]) = 0,
\end{displaymath}
making \(Z_1\) and \(Z_2\) non-trivially L-equivalent. 

Note, if we have a roof of type \(A_r \times A_r, \, A_r^M, \, A_{2r}^G\)
or \(D_r\) then we automatically have \([F_1] = [F_2]\)
since the bases are isomorphic. In \cite{ito:2018},
it is shown that, in the type \(G_2\) case,
we also have \([F_1] = [F_2]\). We will prove
this is also true in the \(C_{3r-1}\) and \(F_4\) cases.

Also observe that if \(M\) is a roof of type \(A_r^M\)
then the \(F_i\) are both isomorphic to \(\P^r\)
and the \(Z_i\) have codimension \(r\).
Thus, the \(Z_i\) are a collection of points and hence
\([Z_i] = \#Z_i\). Therefore, we can only
have \(\L^r([Z_1]-[Z_2]) = 0\) if \([Z_1]=[Z_2]\),
that is roofs of this type do not yield non-trivial
L-equivalences.

For the other cases, in order to show \([Z_1] \neq [Z_2]\), we
will need a few lemmas. First, we must get a grasp
on the Picard groups of the Calabi-Yau pairs.

\begin{lemma}\label[lemma]{lem:pic_restrict}
  Let \(M\) be a homogeneous roof not of type \(A_r \times A_r, \, A_r^M, C_2\) or \( D_4 \).
  Let \(F_i\) be the bases and \((Z_1, Z_2)\) an associated Calabi-Yau
  pair. Then the inclusion map \(Z_i \hookrightarrow F_i \) induces
  an isomorphism on Picard groups, sending ample generator
  to ample generator.
\end{lemma}
\begin{proof}
  Let \( \map{\iota_i}{Z_i}{F_i}\) be the inclusion map.
  This induces a map \(\map{{\iota_i}^\ast}{\Pic(F_i)}{\Pic(Z_i)}\).
  Recall that \(Z_i\) is the zero locus of a general section
  of the ample vector bundle \(\mathcal{E}_i\). By Lefschetz, 
  if \(\dim F_i - \rank \mathcal{E}_i > 2\)
  then \(\iota_i^\ast\) is an isomorphism \cite[Example 7.1.5]{lazarsfeld:2004b}.
  Moreover, as \(\iota_i\) is a closed immersion, \(\iota_i^\ast\OO_{i}(1) = \OO_i(1)\vert_{Z_i}\)
  is the ample generator of \(\Pic(Z_i)\), which we denote by \(\OO_{Z_i}(1)\).
  By consulting the table below, we see the dimensional constraints
  are satisfied whenever \(M\) is not one of type \(A_r \times A_r\),
  \(A_r^M\), \(C_2\) or \(D_4\).

  \begin{center}
      \begin{tabular}{| c | c | c | c |}
        Type & \(\dim F_i\) & \(\rank \mathcal{E}_i\)\\
          \hline
              &&& \\
          \( A_r \times A_r \) & \(r\) & \(r+1\) \\[1em]
          \( A_r^M \) & \(r\) & \( r \) \\[1em]
          \( A_{2r}^G \) & \( r^2+r \) & \( r+1 \) \\[1em]
          \( C_{3r-1} \) & \( 6r^2-3r \) & \( 2r \) \\[1em]
          \( D_r \) & \( r(r-1)/2 \) & \( r \) \\[1em]
          \( F_4 \) & \( 20 \) & \( 3 \) \\[1em]
          \( G_2 \) & \( 5 \) & \( 2 \) \\
      \end{tabular}
  \end{center}
\end{proof}

This lemma will be useful in proving when \([Z_1]\neq[Z_2]\)
due to the following result of Liu and Sebag:

\begin{theorem}[{\cite[Corollary 5]{liu_sebag:2010}}]
  Let \(X\) and \(Y\) be smooth, connected, projective varieties.
  If \([X]=[Y]\), then \(X\) and \(Y\) are stably birational.
\end{theorem}

\begin{cor}\label[corollary]{cor:equals_iso}
  Let \(M\) be a homogeneous roof not of type \(A_r \times A_r, \, A_r^M, \) or
  \(D_4\). If \((Z_1, Z_2)\) is a Calabi-Yau pair associated
  to \(M\), then \([Z_1] = [Z_2]\) if and only if \(Z_1 \cong Z_2\).
\end{cor}
\begin{proof}
  By the theorem, if \([Z_1] = [Z_2]\) then \(Z_1\) and \(Z_2\) are
  stably birational. However, the \(Z_i\) are Calabi-Yau
  and of the same dimension, so they are stably birational
  if and only if they are birational \cite[Corollary 1]{liu_sebag:2010}.
  If we show that the Picard number of both \(Z_i\) is \(1\), this
  combined with the fact that they are Calabi-Yau tells us that
  \(Z_1\) and \(Z_2\) are birational if and only if they are isomorphic 
  \cite[\S 1.1]{chen:2013}. Thus, to prove the result, we need
  only show the Picard number of both the \(Z_i\) is \(1\).

  This follows immediately from \Cref{lem:pic_restrict} for
  all cases under consideration except when \(M\) is of type
  \(C_2\). However, in this case, by \Cref{lem:Z_calabi},
  the \(Z_i\) are Calabi-Yau curves and hence have Picard number \(1\).
\end{proof}

With this in hand, we will prove that \Cref{prop:the_relation}
yields a non-trivial L-equivalence for roofs of type \(C_{3r-1}\) with \(r \geq 2\)
and the roof of type \(F_4\).

\section{L-equivalences via Symplectic Grassmannians}\label{sec:symp}

Fix \(r \geq 2\) and put \(G = \Sp{6r-2}\). Let \(Q\) be the parabolic
obtained by omitting the simple roots \(\alpha_{2r-1}\) and \(\alpha_{2r}\).
Then \(G/Q\) is a homogeneous roof of rank \(2r\) and type \(C_{3r-1}\).
Let \(P_1\) be the parabolic obtained by omitting \(\alpha_{2r-1}\)
and \(P_2\) from omitting \(\alpha_{2r}\). Then put
\begin{align*}
  M &= G/Q = \Iflag{2r-1, 2r}{6r-2}, \\
  F_1 &= G/P_1 = \IGr{2r-1}{6r-2}, \\
  F_2 &= G/P_2 = \IGr{2r}{6r-2}.
\end{align*}
The maps \(\map{q_i}{M}{F_i}\) give \(M\) the structure of
a homogeneous roof.

Write \(\OO_i\) for \(\OO_{P_i}\). Put \(\mathcal{E} = q_1^\ast \OO_1(1) \otimes q_2^\ast \OO_2(1)\)
and \(\mathcal{E}_i = {q_i}_\ast \mathcal{E}\).
Note that \(\mathcal{E}_i = E_{P_i}(\omega_{2r-1}+\omega_{2r})\) by
\Cref{rmk:bwb}.
They are rank \(2r\) bundles and \(M \cong \P(\mathcal{E}_i)\)
by \Cref{prop:proj_nice}.

Let \((Z_1, Z_2)\) be a Calabi-Yau pair associated with \(M\).
By \Cref{prop:the_relation},
we have
\begin{displaymath}
  [\P^{2r-2}]([F_2]-[F_1]) = \L^{2r-1}([Z_1]-[Z_2]).
\end{displaymath}
Our aim is to show that \([F_1] = [F_2]\) and \([Z_1] \neq [Z_2]\), hence
showing that \(Z_1\) and \(Z_2\) are non-trivally L-equivalent.

\begin{lemma}\label[lemma]{lem:equal_fano}
  We have \([F_1] = [F_2]\).
\end{lemma}
\begin{proof}
  Due to the cellular decomposition of \(F_i\), we know
  \begin{displaymath}
    [F_i] = \sum_{j=0}^{\infty} b_{2j}(F_i) \L^j,
  \end{displaymath}
  where \(b_{j}(F_i)\) is the \(j\)th Betti number of \(F_i\).
  We will show the Betti numbers for \(F_1\) and \(F_2\)
  are equal using the Weil conjectures.

  Put \(X = \IGr{d}{2n}\). 
  Let us count \(\# X(\F_q)\). A \(d\)-dimensional isotropic subspace of \(\F_q^{2n}\) is given by a choice of \(d\) basis vectors.
  The first of these can be any non-zero vector and so we have \(q^{2n}-1\) choices. The second vector must be in the
  symplectic complement of the first and not a multiple of the first. This gives \(q^{2n-1} - q\) choices.
  Repeating this, we see that there are
  \begin{displaymath}
    \prod_{i=1}^d (q^{2n -i + 1} - q^{i-1}) = q^{\frac{d(d-1)}{2}} \prod_{i=1}^d (q^{2(n-i) + 2}-1)
  \end{displaymath}
  sets of \(d\) vectors which yield a \(d\)-dimensional isotropic subspace.
  Some of these will result in the same subspace, and so we must count the number of possible bases of a \(d\)-dimensional isotropic subspace.
  There are \(q^d-1\) choices for the first basis vector, then the second must not lie in the span of the first, so there are \(q^d-q\) choices
  for the second basis vector and so on. Hence, the number of possible bases for a given subspace is
  \begin{displaymath}
    \prod_{i=1}^d (q^d - q^{i-1}) = q^{\frac{d(d-1)}{2}} \prod_{i=1}^d (q^i - 1).
  \end{displaymath}
  It follows that
  \begin{displaymath}
    \# X(\F_q) = \prod_{j=1}^d \frac{q^{2(n-j+1)}-1}{q^j-1}.
  \end{displaymath}
  Hence,
  \begin{align*}
    \# F_2(\F_q) &= \prod_{j=1}^{2r} \frac{q^{2(3r-j)}-1}{q^j-1} \\
                 &= \frac{q^{2r}-1}{q^{2r}-1} \prod_{j=1}^{2r-1} \frac{q^{2(3r-j)-1}}{q^j-1} \\
                 &= \# F_1(\F_q).
  \end{align*}
  
  By the Weil conjectures, \([F_1]=[F_2]\).
\end{proof}

Up to some cohomology calculations, which will be postponed
to the next section, we are now able to prove \Cref{th:type_c_zero}.

\begin{proof}[Proof of \Cref{th:type_c_zero}]
  We have the identity
  \begin{displaymath}
    [\P^{2r-2}]([F_2]-[F_1]) = \L^{2r-1}([Z_1]-[Z_2]).
  \end{displaymath}
  By \Cref{lem:equal_fano}, this simplifies to
  \begin{displaymath}
    \L^{2r-1}([Z_1]-[Z_2]) = 0.
  \end{displaymath}
  For this to give us a non-trivial L-equivalence, we 
  need \([Z_1] \neq [Z_2]\) but, by \Cref{cor:equals_iso}, it
  is enough to show \(Z_1\) and \(Z_2\) are not isomorphic. As
  any isomorphism would send \(\OO_{Z_1}(1)\) to \(\OO_{Z_2}(1)\),
  the fact that \(H^\ast(Z_1, \OO_{Z_1}(1))\) and \(H^\ast(Z_2, \OO_{Z_2}(1))\)
  are not isomorphic by \Cref{lem:cohoms_differ} completes the proof.
\end{proof}

\subsection{Cohomology Calculations}

\begin{lemma}\label[lemma]{lem:calc_one}
  For \(p \in \{1, \ldots, 2r\}\), we have \(H^\ast(F_1, \wedge^p \dual{\mathcal{E}}_1 \otimes \OO_1(1))=0\).
\end{lemma}
\begin{proof}
  Recall that
  \begin{displaymath}
    \mathcal{E}_1 = E_{P_1}\left(2\sum_{j=1}^{2r-1}L_j + L_{2r}\right).
  \end{displaymath}
  Thus,
  \begin{displaymath}
    \dual{\mathcal{E}}_1 = E_{P_1}(V_{P_1}(\lambda)),
  \end{displaymath}
  where \(\lambda = 2\sum_{j=1}^{2r-1}L_j + L_{2r}\).
  Put \(V = V_{P_1}(\lambda)\).
  Note that the weights of \(\wedge^p V\)
  are the \(p\)-fold sums of the weights of \(V\)
  so let us calculate them.
  The Weyl group \(W_1\) of \(P_1\) is the subgroup of the Weyl group \(W\)
  of \(G\) generated by \(\{s_1, \ldots, s_{3r-1}\} \setminus \{s_{2r-1}\}\).
  By letting \(W_1\) act on \(\lambda\), we see the extremal
  weights of \(V\) are of the form
  \begin{displaymath}
    2\sum_{j=1}^{2r-1}L_j \pm L_{2r+k}
  \end{displaymath}
  for some \(k \in \{0, \ldots, r-1\}\).
  There are \(2r\) of these, the same as the dimension
  of \(V\), so these are all the weights.

  Hence, we deduce that the only weights of \(\wedge^p V\)
  which are dominant for \(P_1\) are of the form
  \begin{displaymath}
    \mu_0 = 2p\sum_{j=1}^{2r-1}L_j \text{ or } \mu_t = 2p\sum_{j=1}^{2r-1}L_j + \sum_{j=1}^t L_{2r-1+j}.
  \end{displaymath}
  As the unipotent radical of \(P\) acts trivially on \(V\),
  we should be able to take some weights \(\mu_{t_1}, \ldots, \mu_{t_s}\)
  of this form and write
  \begin{displaymath}
    \wedge^p V = V_{P_1}(\mu_{t_1}) \oplus \cdots \oplus V_{P_1}(\mu_{t_s}).
  \end{displaymath}
  Hence, \(\wedge^p \dual{\mathcal{E}}_1 = E_{P_1}(-w_0(\mu_{t_1})) \oplus \cdots \oplus E_{P_1}(-w_0(\mu_{t_s}))\),
  where \(w_0\) is the longest element of \(W_1\). Since \(W_1\)
  is isomorphic to \(S_{2r-1} \times W(\Sp{2r})\), its longest element
  swaps \(L_j\) and \(L_{2r-j}\) for \(j \in \{1, \ldots, 2r-1\}\)
  and sends \(L_j\) to \(-L_j\) for \(j \in \{2r, \ldots, 3r-1\}\).
  Thus, 
  \begin{displaymath}
    -w_0(\mu_t) = - 2p\sum_{j=1}^{2r-1} L_j + \sum_{j=1}^t L_{2r-1+j}.
  \end{displaymath}
  So, if we let
  \begin{displaymath}
    \mu^\prime_t = (1-2p)\sum_{j=1}^{2r-1} L_j + \sum_{j=1}^t L_{2r-1+j},
  \end{displaymath}
  then we see 
  \begin{displaymath}
    \wedge^p \dual{\mathcal{E}}_1 \otimes \OO_1(1) 
    = E_{P_1}(\mu^\prime_{t_1}) \oplus \cdots \oplus E_{P_1}(\mu^\prime_{t_s}).
  \end{displaymath}
  This allows us to calculate its cohomology using Borel-Weil-Bott.
  We claim that, for any \(t\), there is no \(w\in W\) such that
  \(w \bullet \mu^\prime_t\) is dominant for \(G\), which will prove the
  cohomology vanishes. 
  Let \(\rho\) be the sum of
  the fundamental weights and recall
  \begin{displaymath}
    \rho = \sum_{t=1}^{3r-1} (3r-2 + j)L_j.
  \end{displaymath}
  For \(w \bullet \mu^\prime_t\) to be dominant, when writing
  it in the \(L_j\) basis, it can't have any negative coefficients.
  Since the action of \(w\) permutes the \(L_j\) and changes
  some of their signs, to not have negative coefficients in \(w \bullet \mu_t^\prime\),
  there must be a coefficient of some \(L_j\) in
  \(\mu^\prime_t + \rho\) with absolute value at least \(3r-1\).
  A straightforward calculation shows this is not possible if \(p\in\{1, \ldots,2r\}\).
  It follows that \(H^\ast(F_1, \wedge^p \dual{\mathcal{E}}_1 \otimes \OO_1(1)) = 0\).
\end{proof}

\begin{cor}\label[corollary]{cor:dim_Z1}
  We have \(\dim H^0(Z_1, \OO_{Z_1}(1)) = \dim V_{G}(\omega_{2r-1})\)
  while all other cohomologies vanish.
\end{cor}
\begin{proof}
  Recall that \(\OO_1(1)\vert_{Z_1} = \OO_{Z_1}(1)\) by 
  \Cref{lem:pic_restrict}. Thus we have the Koszul resolution:
  \begin{equation*}
    0 \rightarrow \wedge^{2r} \dual{\mathcal{E}}_1 \otimes \OO_1(1) 
    \rightarrow \wedge^{2r-1} \dual{\mathcal{E}}_1 \otimes \OO_1(1) 
    \rightarrow \cdots
    \rightarrow \dual{\mathcal{E}}_1 \otimes \OO_1(1) 
    \rightarrow \OO_1(1) 
    \rightarrow \OO_{Z_1}(1) \rightarrow 0.
  \end{equation*}
  It yields the spectral sequence
  \begin{displaymath}
    E_1^{-q,p} = H^p(F_1, \wedge^p \dual{\mathcal{E}}_1 \otimes \OO_{1}(1)) \implies H^{p-q}(Z_1, \OO_{Z_1}(1)).
  \end{displaymath}
  By the previous result, the only column on the first page which
  doesn't vanish is the cohomology of \(\OO_{1}(1)\). By Borel-Weil-Bott,
  \(H^0(F_1, \OO_{1}(1)) \cong \dual{V_G(\omega_{2r-1})}\) and so
  the spectral sequence implies
  \begin{displaymath}
    \dim H^0(Z_1, \OO_{1}(1)\vert_{Z_1}) = \dim V_G(\omega_{2r-1}),
  \end{displaymath}
  while all other cohomologies vanish.
\end{proof}

\begin{lemma}
  For \(p \in \{1, \ldots, 2r\}\), we have \(H^\ast(F_2, \wedge^p\mathcal{E}_2 \otimes \OO_2(1))=0\).
\end{lemma}
\begin{proof}
  Recall,
  \begin{displaymath}
    \mathcal{E}_2 = E_{P_2}\left(2\sum_{j=1}^{2r-1} L_j + L_{2r}\right).
  \end{displaymath}
  So put \(\lambda = 2\sum_{j=1}^{2r-1} L_j + L_{2r}\) and \(V = V_{P_2}(\lambda)\),
  meaning \(\dual{\mathcal{E}}_2 = E_{P_2}(V)\).
  The Weyl group \(W_2\) of \(P_2\) is the subgroup of \(W\)
  generated by \(\{s_1,\ldots, s_{3r-1}\} \setminus \{s_{2r}\}\).
  Thus, the extremal weights of \(V\) are of the form
  \begin{displaymath}
    2\sum_{j=1}^{2r} L_j -L_k,
  \end{displaymath}
  with \(k \in\{1,\ldots,2r\}\). Since there are \(2r\) of these,
  these make up all the weights of \(V\).
  From this, we deduce that the only dominant weight of
  \(\wedge^p V\) is
  \begin{displaymath}
    \mu_p = 2p\sum_{j=1}^{2r-p} L_j + (2p-1) \sum_{j=2r-p+1}^{2r} L_j,
  \end{displaymath}
  and thus \(\wedge^p V = V_{P_2}(\mu_p)\). Hence, \(\wedge^p \dual{\mathcal{E}}_2 = E_{P_2}(-w_0(\mu_p))\),
  where \(w_0\) is the longest element of \(W_2\). Since \(W_2\) is 
  isomorphic to \(S_{2r} \times W(\Sp{2r-2})\), we see that
  \(w_0\) swaps \( L_j \) and \(L_{2r-j+1}\) for \(j\in\{1,\ldots,2r\}\)
  and sends \(L_j\) to \(-L_j\) for \(j\in\{2r+1,\ldots, 3r-1\}\).
  So,
  \begin{displaymath}
    -w_0(\mu_p) = -(2p-1)\sum_{j=1}^p L_j - 2p\sum_{j=p+1}^{2r}L_j.
  \end{displaymath}
  We thus write
  \begin{displaymath}
    \mu^\prime_p = -(2p-2)\sum_{j=1}^p L_j - (2p-1)\sum_{j=p+1}^{2r}L_j,
  \end{displaymath}
  so that \(\wedge^p \dual{\mathcal{E}}_2 \otimes \OO_2(1) = E_{P_2}(\mu^\prime_p)\).
  Let \(\rho\) be the sum of the fundamental weights of \(G\).
  As explained in the proof of \Cref{lem:calc_one},
  for there to exist \(w \in W\) such that \(w \bullet \mu^\prime_p\)
  is dominant, there must be a coefficient in \(\mu^\prime_p + \rho\)
  with absolute value at least \(3r-1\), when written in
  the \(L_j\) basis. A quick
  calculation shows this is not possible for \(p>1\) and so,
  by Borel-Weil-Bott, \(H^\ast(F_2, \wedge^p\dual{\mathcal{E}}_2 \otimes \OO_2(1)) = 0\)
  for \(p>1\). When \(p=1\),  see that \(\mu_p^\prime = -\sum_{j=2}^{2r} L_j\),
  so while the coefficient of \(L_1\) in \(\mu_p^\prime + \rho\) 
  does have absolute value \(3r-1\), no other
  \(L_j\) has coefficient with absolute value at
  least \(3r-2\). Thus, we still have
  \(H^\ast(F_2, \wedge^p\dual{\mathcal{E}}_2 \otimes \OO_2(1)) = 0\)
  for \(p=1\).
\end{proof}

\begin{cor}\label[corollary]{cor:dim_Z2}
  We have \(\dim H^0(Z_2, \OO_{2}(1)\vert_{Z_2}) = \dim V_G(\omega_{2r})\)
  while all other cohomologies vanish.
\end{cor}
\begin{proof}
  This is analogous to \Cref{cor:dim_Z1}.
\end{proof}

\begin{lemma}\label[lemma]{lem:cohoms_differ}
  We have that \(H^\ast(Z_1, \OO_1(1))\) is not isomorphic
  to \(H^\ast(Z_2, \OO_2(1))\).
\end{lemma}
\begin{proof}
  By \Cref{cor:dim_Z1} and \Cref{cor:dim_Z2},
  it is enough to show \(\dim V_G(\omega_{2r-1}) \neq \dim V_G(\omega_{2r})\).

  It is a standard fact \cite[Theorem 17.5]{fulton_harris:1991} that \(V_G(\omega_k)\) is the kernel
  of the surjective map \(\map{\Lambda}{\wedge^k \C^{6r-2}}{\wedge^{k-2} \C^{6r-2}} \)
  defined by
  \begin{displaymath}
    \Lambda(v_1 \wedge \cdots \wedge v_k) = \sum_{i<j} \tau(v_i,v_j)(-1)^{i+j-1} v_1 \wedge \cdots \wedge \hat{v}_i \wedge \cdots \wedge \hat{v}_j \wedge \cdots \wedge v_k,
  \end{displaymath}
  where \(\tau\) is the standard non-degenerate, alternating bilinear form on \(\C^{6r-2}\).
  Hence, \(\dim V_G(\omega_1) = 6r-2\) while, for \(k > 1\),
  \begin{displaymath}
    \dim V_G(\omega_k) = \binom{6r-2}{k} - \binom{6r-2}{k-2}.
  \end{displaymath}
  Using this, we find that \(\dim V_G(\omega_{2r}) \neq \dim V_G(\omega_{2r-1})\),
  proving the result.
\end{proof}

\section{\texorpdfstring{L-equivalences via \(F_4\) Grassmannians}{L-equivalences via F4 Grassmannians}}\label{sec:f4}

When \(G\) is the semisimple, simply connected group of type \(F_4\)
and \(Q, \, P_1\) and \(P_2\) the parabolics corresponding to the
Dynkin diagrams \dynkin F{*xx*}, \dynkin F{*x**} and \dynkin F{**x*}
respectively, then \(M=G/Q\) is a homogeneous roof of rank \(3\)
with projective bundle structures given by the maps
\(M \to F_i\) where \(F_i = G/P_i\).

We construct \(\mathcal{E}\) and the \( \mathcal{E}_i \) as in \Cref{section:homog}.
Recall that \(\mathcal{E}\) is rank \(1\) and \(\mathcal{E}_i\) is 
rank \(3\). 
Let \((Z_1, Z_2)\) be a Calabi-Yau pair associated to \(M\).
Using \Cref{prop:the_relation}, we get the identity
\begin{displaymath}
  [\P^1]([F_2]-[F_1]) = \L^2([Z_1]-[Z_2]).
\end{displaymath}
Again, we want to show \([F_1] =[F_2]\) and that \([Z_1] \neq [Z_2]\)
by showing \(Z_1 \not\cong Z_2\) and using \Cref{cor:equals_iso}.

\begin{lemma}
  We have \([F_1] = [F_2]\).
\end{lemma}
\begin{proof}
  Let \(W_i\) be the Weyl group for \(P_i\) as a subgroup of \(W\),
  the Weyl group of \(G\).
  Using the cellular decomposition of \(F_i\), we can check if
  \([F_1] = [F_2]\) by calculating the lengths of the
  shortest representatives of \(W/W_i\). We do this using
  Sage \cite{sage} and the computations are recorded in
  the ancillary files. With these, we find
  \begin{align*}
    [F_1] = [F_2] =  1 + \L &+ 2\L^2 + 3\L^3 + 4\L^4 + 5\L^5 + 6\L^6
                      + 7\L^7 + 7\L^8 + 8\L^9 + 8 \L^{10} \\ 
                      &+ 8\L^{11} + 7\L^{12} + 7\L^{13} 
                      + 6\L^{14} + 5\L^{15} + 4\L^{16} + 3\L^{17} + 2\L^{18} + \L^{19} + \L^{20}.
  \end{align*}
\end{proof}

\begin{lemma}
  We have \(Z_1 \not \cong Z_2\).
\end{lemma}
\begin{proof}
  Again, recall that \(\OO_i(1)\vert_{Z_i} = \OO_{Z_i}(1)\) by 
  \Cref{lem:pic_restrict}. Thus we have the Koszul resolution:
  \begin{displaymath}
    0 \rightarrow \wedge^3 \dual{\mathcal{E}}_i \otimes \OO_i(1)
    \rightarrow \wedge^2 \dual{\mathcal{E}}_i \otimes \OO_i(1)
    \rightarrow \dual{\mathcal{E}}_i \otimes \OO_i(1)
    \rightarrow \OO_i(1) 
    \rightarrow \OO_{Z_i}(1)
    \rightarrow 0.
  \end{displaymath}
  This allows us to calculate \(H^\ast(Z_i, \OO_{Z_i}(1))\),
  which we can use to show \(Z_1 \not\cong Z_2\).

  The Weyl group \(W\) of \(G\) is generated by
  \(\{s_1, s_2, s_3, s_4\}\) and their action on
  the fundamental weights is given by
  \begin{align*}
      &s_1(\omega_1) = -\omega_1 + \omega_2, \hspace{3.6em} s_2(\omega_2) = \omega_1 -
      \omega_2 + 2\omega_3, \\ 
      &s_3(\omega_3) = \omega_2 - \omega_3 + \omega_4, \hspace{2em} s_4(\omega_4)
      = \omega_3 - \omega_4.
  \end{align*}
  Hence, if \(W_1\) is the Weyl group of \(P_1\)
  then
  \begin{displaymath}
    W_1 = \{\id, s_1, s_3, s_4, s_1s_3, s_1s_4, s_3s_4, s_4s_3, s_1s_3s_4, s_1s_4s_3,
    s_3s_4s_3, s_1s_3s_4s_3\},
  \end{displaymath}
  which has longest element \(s_1s_3s_4s_3\). Then
  \begin{displaymath}
    \dual{\mathcal{E}}_1 = E_{P_1}(-s_1s_3s_4s_3(\omega_2 + \omega_3)) = E_{P_1}(-2\omega_2 + \omega_4).
  \end{displaymath}
  The orbit of \(-2\omega_2 + \omega_4\) under \(W_1\) is
  \begin{displaymath}
    \{-2 \omega_2 + \omega_4, \, -2\omega_2 + \omega_3 - \omega_4, \, -\omega_2
    - \omega_3\},
  \end{displaymath}
  so this must be all the weights of \(\dual{\mathcal{E}_1}\), each occurring with multiplicity one.
  By taking sums of these, we see that the weights of \( \wedge^2
  \dual{\mathcal{E}_1} \) are
  \begin{displaymath}
      \{-4\omega_2 + \omega_3, \, -3\omega_2 - \omega_3 + \omega_4, \, -3\omega_2 - \omega_4\}.
  \end{displaymath}
  The only one of these which is dominant with respect to \( P_1 \) is \(
  -4\omega_2 + \omega_3 \), so we must have
  \begin{displaymath}
      \wedge^2 \dual{\mathcal{E}_1} = E_{P_1}(-4\omega_2 + \omega_3).
  \end{displaymath}
  Summing weights again, we get that
  \begin{displaymath}
      \wedge^3 \dual{\mathcal{E}_1} = E_{P_1}(-5\omega_2).
  \end{displaymath}
  After tensoring everything by \( \OO_{P_1}(1) \), 
  the Koszul resolution yields an exact sequence
  \begin{displaymath}
      0 \rightarrow E_{P_1}(-4 \omega_2) \rightarrow E_{P_1}(-3\omega_2 +
      \omega_3) \rightarrow E_{P_1}(-\omega_2 + \omega_4)
      \rightarrow \OO_{P_1}(1) \rightarrow \OO_{Z_1}(1) \rightarrow 0
  \end{displaymath}
  To compute the cohomologies of these bundles, we again use Borel-Weil-Bott. The
  most straightforward way to do this is using a computer, for instance using
  Rampazzo's online calculator at \cite{rampazzo_calc}.
  We find that the cohomologies of \( E_{P_1}(-4\omega_2),
  E_{P_1}(-3\omega_2 + \omega_3) \) and \( E_{P_1}(-\omega_2 + \omega_4) \)
  vanish in every degree, so we conclude that \( H^i(Z_1, \OO_{Z_1}(1)) = 0 \)
  for \( i > 0 \), while
  \begin{displaymath}
      \dim H^0(Z_1, \OO_{Z_1}(1)) = \dim H^0(F_1, \OO_{P_1}(1)) = \dim
      V_G(\omega_2) = 1274.
  \end{displaymath}

  We can do something analogous for \(Z_2\) which will tell us that
  \begin{displaymath}
    \dim H^0(Z_2, \OO_{Z_2}(1)) = \dim V_G(\omega_3) = 273.
  \end{displaymath}

  Hence, \(Z_1 \not\cong Z_2\).
\end{proof}

\begin{proof}[Proof of \Cref{th:type_f4_zero}]
  We can use \Cref{cor:equals_iso} in this case so, as \(Z_1 \not \cong Z_2\),
  we know that \([Z_1]\neq [Z_2]\). Combining this with \([F_1] = [F_2]\)
  and using \Cref{prop:the_relation} yields
  \begin{displaymath}
    \L^2([Z_1]-[Z_2]) = 0.
  \end{displaymath}
\end{proof}

\bibliographystyle{plain}
\bibliography{refs}

\begin{thebibliography}{10}

\bibitem{baston_eastwood:1989}
Robert~J. Baston and Michael~G. Eastwood.
\newblock {\em The {Penrose} Transform: Its Interactions with Representation Theory}.
\newblock Oxford Mathematical Monographs. Oxford University Press, 1989.

\bibitem{borel:1991}
Armand Borel.
\newblock {\em Linear Algebraic Groups}.
\newblock Number 126 in Graduate Texts in Mathematics. Springer, 2 edition, 1991.

\bibitem{borisov:2017}
Lev~A Borisov.
\newblock The class of the affine line is a zero divisor in the {Grothendieck} ring.
\newblock {\em J. Algebraic Geom.}, 27(2):203--209, June 2017.

\bibitem{chen:2013}
Xi~Chen.
\newblock Rational self maps of {Calabi-Yau} manifolds.
\newblock {\em Clay Mathematics Proceedings}, 18:171--184, 2013.

\bibitem{sage}
W.\thinspace{}A.~Stein et. al.
\newblock {\em {S}age {M}athematics {S}oftware ({V}ersion 10.7)}.
\newblock The Sage Development Team, 2025.
\newblock \url{https://www.sagemath.org/}.

\bibitem{fulton_harris:1991}
William Fulton and Joe Harris.
\newblock {\em Representation Theory}.
\newblock Number 129 in Graduate Texts in Mathematics. Springer, 1991.

\bibitem{humphreys:1975}
James~E. Humphreys.
\newblock {\em Linear Algebraic Groups}.
\newblock Number~21 in Graduate Texts in Mathematics. Springer, 1975.

\bibitem{ito:2018}
Atsushi Ito, Makoto Miura, Shinnosuke Okawa, and Kazushi Ueda.
\newblock The class of the affine line is a zero divisor in the {Grothendieck} ring: Via {$G_2$-Grassmannians}.
\newblock {\em J. Algebraic Geom.}, 28(2):245--250, December 2018.

\bibitem{kanemitsu:2018}
Akihiro Kanemitsu.
\newblock Mukai pairs and simple {$K$}-equivalence, December 2018.
\newblock \url{https://arxiv.org/abs/1812.05392} [math.AG].

\bibitem{kock:1991}
Bernhard K{\"{o}}ck.
\newblock {Chow} motif and higher {Chow} theory of {G/P}.
\newblock {\em Manuscripta Mathematica}, 70(4):363--372, 1991.

\bibitem{kuznetsov:2018}
Alexander Kuznetsov.
\newblock Derived equivalence of {Ito-Miura-Okawa-Ueda Calabi-Yau} 3-folds.
\newblock {\em Journal of the Mathematical Society of Japan}, 70(3):1007--1013, July.

\bibitem{kuznetsov_shinder:2018}
Alexander Kuznetsov and Evgeny Shinder.
\newblock {Grothendieck} ring of varieties, {D}- and {L}-equivalence, and families of quadrics.
\newblock {\em Selecta Mathematica}, 24:3475--3500, 2018.

\bibitem{lazarsfeld:2004b}
Robert Lazarsfeld.
\newblock {\em Positivity in Algebraic Geometry II: Positivity for Vector Bundles, and Multiplier Ideals}.
\newblock Number~49 in A Series of Modern Surveys in Mathematics. Springer, 2004.

\bibitem{liu_sebag:2010}
Qing Liu and Julien Sebag.
\newblock The {Grothendieck} ring of varieties and piecewise isomorphisms.
\newblock {\em Mathematische Zeitschrift}, 265:321--342, 2010.

\bibitem{meinsma:2025}
Reinder Meinsma.
\newblock Counterexamples to the {Kuznetsov--Shinder} {L}-equivalence conjecture, 2025.
\newblock \url{https://arxiv.org/abs/2503.21511} [math.AG].

\bibitem{milne:2017}
J.~S. Milne.
\newblock {\em Algebraic Groups}.
\newblock Number 170 in Cambridge Studies in Advanced Mathematics. Cambridge University Press, 2017.

\bibitem{poonen:2002}
Bjorn Poonen.
\newblock The {Grothendieck} ring of varieties is not a domain.
\newblock {\em Mathematical Research Letters}, 9(4):493--497, 2002.

\bibitem{rampazzo_calc}
Marco Rampazzo.
\newblock {Borel-Weil-Bott Calculator}.
\newblock Accessed December 2025: \url{https://borelweilbott.marcorampazzo.com/ }.

\bibitem{rampazzo:2021}
Marco Rampazzo.
\newblock {\em Equivalences Between {Calabi–Yau} Manifolds and Roofs of Projective Bundles}.
\newblock Phd thesis, University of Stavanger, April 2021.

\bibitem{rampazzo:2022}
Marco Rampazzo.
\newblock New counterexamples to the birational {Torelli} theorem for { Calabi--Yau} manifolds, 2022.
\newblock \url{https://arxiv.org/abs/2211.03702} [math.AG].

\bibitem{rampazzo_kapustka:2019}
Marco Rampazzo and Micha\l~\. Kapustka.
\newblock Torelli problem for {Calabi--Yau} threefolds with {GLSM} description, 2019.
\newblock \url{https://arxiv.org/abs/1711.10231} [math.AG].

\bibitem{rampazzo_xie:2025}
Marco Rampazzo and Ying Xie.
\newblock Derived equivalence for the simple flop of type {$D_5$}, 2025.
\newblock \url{https://arxiv.org/abs/2410.20446} [math.AG].

\bibitem{segal:2016}
Ed~Segal.
\newblock A new 5-fold flop and derived equivalence.
\newblock {\em Bulletin of the London Mathematical Society}, 48(3):533--538, 2016.

\bibitem{snow}
Dennis~M. Snow.
\newblock Homogeneous vector bundles.
\newblock Retrieved from the Internet Archive, accessed November 2025: \url{ https://web.archive.org/web/20241009004944/http://www.nd.edu/~ snow/Papers/HomogVB.pdf}.

\bibitem{springer:1998}
T.A. Springer.
\newblock {\em Linear Algebraic Groups}.
\newblock Modern Birkh{\"{a}}user Classics. Birkh{\"{a}}user, 2 edition, 1998.

\end{thebibliography}

\end{document}